\documentclass[11pt,reqno]{amsart}
\usepackage[utf8]{inputenc}
\usepackage[T1]{fontenc}
\usepackage{amsmath}
\usepackage{amsfonts}
\usepackage{amssymb}
\usepackage{mhchem}
\usepackage{stmaryrd}
\usepackage{bbold}
\usepackage{csquotes}

\MakeOuterQuote{"}
\usepackage{mathtools, nccmath}

\usepackage{graphicx}
\usepackage[export]{adjustbox}
\usepackage{setspace}
\graphicspath{ {./images/} }
\usepackage{CJKutf8}
\newtheorem{thm}{Theorem}
\newtheorem{lem}[thm]{Lemma}
\newtheorem{prop}[thm]{Proposition}

\theoremstyle{definition}

\newcommand{\vertiii}[1]{{\left\vert\kern-0.25ex\left\vert\kern-0.25ex\left\vert
#1 \right\vert\kern-0.25ex\right\vert\kern-0.25ex\right\vert}}

\def \diag  {\text {\rm diag}}
\def \e   {\text {\rm e}}

\begin{document}

\title[]{The numerical radius and positivity of block matrices}

\author[Rajendra Bhatia]{Rajendra Bhatia}
\address{Ashoka University, Sonepat\\ Haryana, 131029, India}

\email{rajendra.bhatia@ashoka.edu.in}

\author{Tanvi Jain}

\address{Indian Statistical Institute, New Delhi 110016, India}

\email{tanvi@isid.ac.in}

\date{\today}

\begin{abstract}
This article has two interpenetrating motifs.
One is an exposition of some major ideas and techniques behind the use of block matrices,
and especially their positivity properties.
This is done by focussing on one major problem:
characterisation of operators whose numerical radius is bounded by one.
So, the article could serve as an introduction to that topic as well.
\end{abstract}

\subjclass[2010]{47A12, 47A20, 47B65}

\keywords{Numerical radius, block matrix, dilation, power inequality, positive operator}

\maketitle

\section{Introduction}

Block matrix techniques have come to play a powerful role in matrix analysis and operator theory. There are many
theorems that give a characterisation of a property of a matrix $A$ (such as its norm) in terms of another property (such as positivity) of a larger matrix in which $A$ sits as a block.
In the theory of the Schur complement,
and in the study of positive linear maps, currently a topic of great interest in quantum information, many important theorems are framed in terms of $2 \times 2$ block matrices. See, e.g., \cite{rbh1} and the survey article
\cite{zh1} in the book \cite{zh}.

The principal aim of this expository article is to show the essence of this technique in the context of a specific problem involving the numerical radius. We emphasize general ideas involving positivity of matrices and trigonometric polynomials
that can be used in a variety of problems.

Let $\mathcal{H}$ be the $d$-dimensional complex Euclidean space $\mathbb{C}^{d}$ (or, more generally, a complex Hilbert space).
We choose the convention that the inner product $\langle x, y\rangle$ on $\mathcal{H}$ is linear in the second variable $y$ and conjugate linear in $x.$
Let $A$ be an operator on $\mathcal{H}.$ By its {\it norm} we mean the operator norm
\begin{equation}
\|A\|=\sup _{\|x\|=1}\|A x\|=\sup _{\|x\|=\|y\|=1}|\langle x, A y\rangle|.\label{eq1}
\end{equation}
The {\it numerical radius} of $A$ is defined as
\begin{equation}
w(A)=\sup _{\|x\|=1}|\langle x, A x\rangle|.\label{eq2}
\end{equation}
The numerical radius is also a norm.
A comparison between this and the operator norm is given by the inequalities:
\begin{equation}
w(A) \le\|A\| \le 2 w(A).\label{eq3}
\end{equation}
The numerical radius is associated with the concept of the {\it numerical range}, (also called the {\it field of values}), a subset of the complex plane defined as
$$
W(A)=\{\langle x, A x\rangle:\|x\|=1\}.
$$
Much information about $A$ is contained in the set $W(A)$ and it is an object of intense study.
See, e.g., \cite{al,gr,hal,hj2,wg}.

The {\it adjoint} of $A$ is the unique operator $A^{*}$ that satisfies the relation
$$\langle x, A y\rangle=\left\langle A^{*} x, y\right\rangle\textrm{ for all }x, y\textrm{ in }\mathcal{H}.$$
If $A$ is {\it self-adjoint} $\left(A=A^{*}\right),$ then $W(A)$ is a subset of the real line, and $w(A)=\|A\|.$
We say that $A$ is {\it positive semidefinite} (just {\it positive}, for brevity) if $\langle x, A x\rangle \ge 0$ for all $x.$ See \cite{hj1} or \cite{rbh1} for basic facts about positive operators.

Having with us the two measures $w(A)$ and $\|A\|$ of the size of $A,$ we may ask the questions: when is $w(A) \leqslant 1$ and when is $\|A\| \leqslant 1$ ?
The two theorems given below answer these questions in terms of positivity of some block matrices constructed from $A.$

\begin{thm}\label{thm1}
The following conditions are equivalent:
\begin{itemize}
\item[(i)] $\|A\| \leqslant 1.$

\item[(ii)] The $2 \times 2$ block matrix $\left[\begin{array}{cc}I & A^{*} \\ A & I\end{array}\right]$ is positive.

\item[(iii)] The $(n+1) \times(n+1)$ block matrix
\begin{equation}
\begin{bmatrix}I & A^* & A^{*2} & \cdots & A^{*n}\\
A & I & A^* & \cdots & A^{*n-1}\\
\vdots & \vdots & \vdots & \vdots\vdots\vdots & \vdots\\
A^{n} & A^{n-1} & A^{n-2} & \cdots & I\end{bmatrix}\label{eq4}
\end{equation}
is positive for every $n.$
\end{itemize}
\end{thm}

\begin{thm}\label{thm2}
The following conditions are equivalent:
\begin{itemize}
\item[(i)] $w(A) \leq 1.$

\item[(ii)] The $(n+1) \times(n+1)$ block matrix
\begin{equation}
\Delta_{n}(A)=\begin{bmatrix}2I & A^{*} & 0 & \cdots & 0 \\ A & 2I & A^{*} & \cdots & 0 \\ 0 & A & 2I & \cdots & 0 \\ \cdot & \cdot & \cdot & \cdots & \cdot \\ 0 & 0 & 0 & \cdots & 2I\end{bmatrix}\label{eq5}
\end{equation}
is positive for every $n.$

\item[(iii)] The $(n+1) \times(n+1)$ block matrix
\begin{equation}
\Gamma_n(A)=\begin{bmatrix}2I & A^* & A^{*2} & \cdots & A^{*n}\\
A & 2I & A^* & \cdots & A^{*n-1}\\
\vdots & \vdots & \vdots & \vdots\vdots\vdots & \vdots\\
A^{n} & A^{n-1} & A^{n-2} & \cdots & 2I\end{bmatrix}\label{eq6}
\end{equation}
is positive for every $n.$ 
\end{itemize}
\end{thm}

The matrices \eqref{eq4} \eqref{eq5} and \eqref{eq6} have a very special form. A matrix $T$ is called a
{\it Toeplitz matrix} if its entries $t_{i j}$ obey the relation $t_{i j}=t_{i-j}.$
This says that on each of the diagonals of $T$ parallel to the main diagonal the same entry is repeated.
Each of the matrices \eqref{eq4}, \eqref{eq5} and \eqref{eq6} is Hermitian and {\it block-Toeplitz}.
Further the matrix in \eqref{eq5} is {\it block-tridiagonal}; all its block entries outside the three middle diagonals are zero.

In Section 2 we give proofs of
Theorems \ref{thm1} and \ref{thm2} using elementary ideas around positivity of matrices and Fourier series.
In Section 3 we show how Theorem \ref{thm2} can be used to prove one of the fundamental properties - the power inequality - of the numerical radius, and recall its role in the development of the subject.
We also indicate how the positivity of the block matrices \eqref{eq4} and \eqref{eq6} leads to two famous {\it dilation theorems}
of major interest and importance.
We take this opportunity to gather and prove several other useful
conditions equivalent to those in Theorem \ref{thm2}.
These are given in the following theorem, due to T. Ando.

\begin{thm}\label{thm3}
The following conditions are equivalent:
\begin{itemize}
\item[(i)] $w(A) \leqslant 1.$

\item[(ii)] There exist operators $X$ and $Y$ such that 
\begin{equation}
X^{*} X+Y^{*} Y=I\textrm{ and }A=2 X^{*} Y.\label{eq7}
\end{equation}

\item[(iii)] There exists an operator $C$ such that
\begin{equation}
\|C\| \leqslant 1 \text { and } A=2\left(I-C^{*} C\right)^{1 / 2} C.\label{eq8}
\end{equation}
\item[(iv)] There exists a Hermitian operator $H$ such that the $2 \times 2$ block matrix
\begin{equation}
\left[\begin{array}{cc}
I+H & A^{*} \\
A & I-H
\end{array}\right]\label{eq9}
\end{equation}
is positive.
\end{itemize}
\end{thm}

While the unifying theme is positivity and block matrices, the proof of
Theorem \ref{thm3} invokes a little more advanced tool than our proofs of Theorems \ref{thm1} and \ref{thm2}.
This is discussed in Section 4.

A standard reference for positivity, and its various ramifications in operator algebras, is Paulsen \cite{paulsen}. The recent book \cite{wg} by Wu and Gau provides an encyclopedic coverage of the numerical range and numerical radius. A lot more beyond what we discuss here can be found in these two books.
Except for some details and points of emphasis, all major ideas in our proofs of Theorems \ref{thm1} and \ref{thm2} are taken from \cite{paulsen} Chapters 1-3.
Our proof of Theorem \ref{thm3} is different from ones we have seen. See the remarks after the proof of the Theorem.

\section{Proofs of Theorems \ref{thm1} and \ref{thm2}}

If $A$ and $B$ are self-adjoint, the notation $A \geqslant B$ means $A-B$ is positive. In particular, $A \geqslant 0$ means $A$ is positive.
From the definition of positivity, it follows that if $A \geqslant 0,$ and $X$ is any operator, then $X^{*} A X \geqslant 0.$ This fact will be used very often.

The condition $\|A\| \leqslant 1$ means $\|A x\|^{2} \leqslant\|x\|^{2}$ for every $x \in \mathcal{H}.$
Writing this as $\langle A x, A x\rangle=\left\langle x, A^{*} A x\right\rangle \leqslant\langle x, x\rangle,$ one sees that $\|A\| \leq 1$ if and only if $I-A^{*} A \geqslant 0.$ Since $\|A\|=\left\|A^{*}\right\|,$ this is equivalent to $I-A A^{*} \geqslant 0.$
A little less obvious is the assertion of the following lemma.

\begin{lem}\label{lem4}
Let $I-A$ be invertible (i.e. the spectrum of $A$ does not contain the point $1$).
Then 
\begin{equation}
\|A\| \leqslant 1 \Leftrightarrow(I-A)^{-1}+\left(I-A^{*}\right)^{-1}-I \geqslant 0.\label{eq10}
\end{equation}
\end{lem}

\begin{proof}
We observe that 
$$
\begin{aligned}
&(I-A)^{-1}+\left(I-A^{*}\right)^{-1}-I \\
&=(I-A)^{-1}\left[\left(I-A^{*}\right)+(I-A)-(I-A)\left(I-A^{*}\right)\right]\left(I-A^{*}\right)^{-1} \\
&=(I-A)^{-1}\left[I-A A^{*}\right]\left(I-A^{*}\right)^{-1},
\end{aligned}
$$
and this is positive if and only if $I-A A^{*} \geqslant 0,$ or equivalently, $\|A\| \leqslant 1.$
\end{proof}

Consider now the $2 \times 2$ block matrix
\begin{equation}
R_{1}(A)=\left[\begin{array}{ll}
0 & 0 \\
A & 0
\end{array}\right].\label{eq11}
\end{equation}
Let $\underline{I}$ denote the identity operator $\left[\begin{array}{ll}I & 0 \\ 0 & I\end{array}\right].$ Then $\underline{I}-R_{1}(A)$ is invertible.
Further
$$
\left[\begin{array}{cc}
I & 0 \\
-A & I
\end{array}\right]^{-1}=\left[\begin{array}{ll}
I & 0 \\
A & I
\end{array}\right].
$$
So, by Lemma \ref{lem4}, $\left\|R_{1}(A)\right\| \leqslant 1$ if and only if 
$$\left[\begin{array}{ll}I & 0 \\ A & I\end{array}\right]+\left[\begin{array}{cc}I & A^{*} \\ 0 & I\end{array}\right]-\left[\begin{array}{cc}I & 0 \\ 0 & I\end{array}\right] \geqslant 0,$$
i.e.,
$$\left[\begin{array}{cc}I & A^{*} \\ A & I\end{array}\right] \geqslant 0.$$
But $\left\|R_{1}(A)\right\|=\|A\|.$ This shows that the conditions (i) and (ii) of Theorem \ref{thm1} are equivalent.

The ideas of this proof can be generalised. Let $R_{n}(A)$ be the $(n+1) \times(n+1)$
block matrix
\begin{equation}
R_{n}(A)=\begin{bmatrix}
0 & 0 & \cdots & 0 & 0\\
A & 0 & \cdots & 0 &0\\
0 & A & \cdots & 0 &0\\
0 & 0 & \cdots & A & 0
\end{bmatrix}\label{eq12}
\end{equation}
with entries $A$ on its first subdiagonal and $0$'s elsewhere.
Then for $1 \leqslant k \leqslant n$ $R_{n}(A)^{k}$ is the block matrix with entries $A$ on the $k\textrm{th}$ subdiagonal and $0$'s elsewhere, and $R_{n}(A)^{n+1}=0.$
So,
$$
\begin{aligned}
\left(\underline{I}-R_{n}(A)\right)^{-1} &=\underline{I}+R_{n}(A)+\cdots+R_{n}(A)^{n} \\
& =\begin{bmatrix}I & 0 & 0 & \cdots & 0 \\
A & I & 0 &\cdots & 0 \\
A^{2} & A & I & \cdots & 0 \\
\vdots & \vdots & \vdots & \vdots\vdots\vdots & \vdots\\
A^{n} & A^{n-1} & A^{n-2} & \cdots& I\end{bmatrix} 
\end{aligned}
$$
Also $\left\|R_{n}(A)\right\|=\|A\|.$ So, by Lemma \ref{lem4},
$\|A\| \leqslant 1$ if and only if
$$
\left(\underline{I}-R_{n}(A)\right)^{-1}+\left(\underline{I}-R_{n}\left(A\right)^{*}\right)^{-1}-\underline{I} \geqq 0.
$$
This is the condition (iii) of Theorem \ref{thm1}. We have thus proved Theorem 1 completely.

We now turn to the proof of Theorem \ref{thm2}.
The equivalence (ii) $\Leftrightarrow$ (iii) is established using ideas very similar to ones used above.

The {\it real part} of an operator $A$ is the Hermitian operator $\operatorname{Re}\, A$ defined as
\begin{equation}
\operatorname{Re} A=\frac{1}{2}\left(A+A^{*}\right).\label{eq13}
\end{equation}

\begin{lem}\label{lem5}
Let $A$ be an invertible operator. Then $\operatorname{Re} A \geqslant 0$ if and only if $\operatorname{Re} A^{-1} \geqslant 0.$
\end{lem}

\begin{proof}
Note that
$A^{-1}\left(\frac{A+A^{*}}{2}\right) A^{*-1}=\frac{A^{*-1}+A^{-1}}{2} .$
So,
$\frac{1}{2}\left(A+A^*\right)$ is positive if and only if $\frac{1}{2}\left(A^{-1}+A^{*-1}\right)$ is positive.
\end{proof}

Now let $R_{n}(A)$ be the matrix defined by \eqref{eq12}.
Apply Lemma \ref{lem5} to the operator $X=\underline{I}-R_{n}(A)$ in place of $A.$ Note that $2 \operatorname{Re} X=\Delta_{n}(-A)$ and $2 \operatorname{Re} X^{-1}=\Gamma_{n}(A),$ where $\Delta_{n}(A)$ and $\Gamma_{n}(A)$ are defined by \eqref{eq5} and \eqref{eq6}. So, these two block matrices are positive together.
Now, if $D$ is the block-diagonal matrix

$$D=\operatorname{diag}\left(-I, I,-I, \ldots,(-1)^{n+1} I\right),$$
then $D \Delta_{n}(A) D=\Delta_{n}(-A).$ So, the block matrices $\Delta_{n}(A)$ and $\Gamma_{n}(A)$ are positive together. This proves the equivalence of conditions (ii) and (iii) in Theorem \ref{thm2}.

We will now prove the implication (i) $\Leftrightarrow$ (ii) of Theorem \ref{thm2}.
For this we need to invoke a matrix version of the Herglotz theorem from Fourier series. Let $f$ be a continuous complex function on $[-\pi, \pi]$ with a convergent Fourier series
\begin{equation}
f(t)=\sum_{n=-\infty}^{\infty} c_{n} \e^{\imath  n t}.\label{eq14}
\end{equation}
The Fourier coefficients $c_{n}$ are given by the formula
\begin{equation}
c_{n}=\frac{1}{2 \pi} \int_{-\pi}^{\pi} f(t) \e^{-\imath  n t} d t .\label{eq15}
\end{equation}
The Herglotz theorem says that $f(t) \geqslant 0$ for all $t$ if and only if the sequence $\left\{c_{n}\right\}_{n \in \mathbb{Z}}$ is a {\it positive definite sequence}. This last statement means that for each $N,$ the $N\times N$ Toeplitz matrix $T_{N}(f)$ with entries $t_{rs}=c_{r-s}$ is positive.
Explicitly, $T_N(f)$ is the matrix
$$
T_{N}(f)=\begin{bmatrix}
c_0 & c_{-1} & c_{-2} & \cdots & c_{-(N-1)}\\
c_1 & c_0 & c_{-1} & \cdots & c_{-(n-2)}\\
\vdots & \vdots & \vdots & \vdots\vdots\vdots & \vdots\\
c_{n-1} & c_{n-2} & c_{n-3} & \cdots & c_0\end{bmatrix}.$$
See, e.g., \cite{rbh2}
Section 3.9 or \cite{ru}.
We need an operator version of this theorem
(in one direction), and we supply
a proof of that.

\begin{lem}\label{lem6}
Let $F(t)$ be an operator-valued function given by the Fourier series
\begin{equation}
F(t)=\sum_{n=-\infty}^{\infty} C_{n} \e^{\imath  n t},\label{eq16}
\end{equation}
and let $T_N(F)$ be the $N\times N$ block Toeplitz matrix whose $r,s$ entry is $C_{r-s}.$
If $F(t) \geqslant 0$ for all $t,$ then for every $N$
the matrix $T_{N}(F)$ is positive.
\end{lem}

\begin{proof}
The coefficients $C_{n}$ are given, as in \eqref{eq15}, by
\begin{equation}
C_{n}=\frac{1}{2 \pi} \int_{-\pi}^{\pi} F(t)\, \e^{-\imath  n t} dt.\label{eq17}
\end{equation}
To show that $T_{N}(F)$ is positive we have to show that $\left\langle\underline{x}, T_{N}(F) \underline{x}\right\rangle \geqslant 0$ for every vector $\underline{x}=\left[\begin{array}{l}x_{1} \\ x_{2} \\ \vdots \\ x_{N}\end{array}\right],$ where $x_{j} \in \mathcal{H}, 1 \leqslant j \leqslant N.$
To see this, note that
$$
\begin{aligned}
\left\langle\underline{x}, T_{N}(F) \underline{x}\right\rangle &=\sum_{r, s}\left\langle x_{r}, C_{r-s} x_{s}\right\rangle \\
&=\frac{1}{2 \pi} \int_{-\pi}^{\pi} \sum_{r, s} \e^{-\imath (r-s) t}\left\langle x_{r}, F(t) x_{s}\right\rangle
\end{aligned}
$$
using \eqref{eq17}.
Let $F(t)^{1 / 2}$ be the unique positive square root of $F(t).$ Then the integrand above can be expressed as
$$
\left\|\sum_{r} \e^{-\imath  r t} F(t)^{1 / 2} x_{r}\right\|^{2}.
$$
This is a nonnegative quantity, and hence so is the integral.
\end{proof}

The link between Toeplitz matrices and the numerical radius is made via the following lemma.

\begin{lem}\label{lem7}
The following conditions are equivalent:
\begin{itemize}
\item[(i)] $w(A) \leq 1$

\item[(ii)] $\operatorname{Re} \e^{\imath  t} A \leqslant I$ for all $t \in \mathbb{R}.$

\item[(iii)] For every $t,$ the operator
\begin{equation}
F(t)=\e^{-\imath  t} A^{*}+2 I+\e^{\imath  t} A\label{eq18}
\end{equation}
is positive.
\end{itemize}
\end{lem}

\begin{proof}
Let $z$ be any complex number. Then $|z| \leqslant 1$ if and only if $\operatorname{Re}\, \e^{\imath  t} z \leqslant 1$ for all $t.$ Hence $w(A) \leq 1$ if and only if
$$\operatorname{Re}\, \e^{\imath  t}\langle x, A x\rangle \leqslant 1$$
for all $x$ with $\|x\|=1.$
This translates to
$$
\left\langle x,\left(\operatorname{Re}\, \e^{\imath  t} A\right) x\right\rangle \leqslant\langle x, x\rangle
$$
for all $x.$ Thus the conditions (i) and (ii) are equivalent. Since
$$
\operatorname{Re}\, \e^{\imath  t} A=\frac{1}{2}\left(\e^{-\imath  t} A^{*}+\e^{\imath  t} A\right),
$$
the condition (iii) is just a restatement of (ii).
\end{proof}
\vskip.1in
\noindent{\bf Proof of (i) $\Rightarrow$ (ii) in Theorem \ref{thm2}}.
 Let $w(A) \leqslant 1.$ Then by Lemma \ref{lem7} the matrix $F(t)$ defined by \eqref{eq18} is positive for all $t.$ By Lemma \ref{lem6} all the
Toeplitz matrices $T_{N}(F)$ are positive. The special form of $F$ shows that all these Toeplitz matrices are block-tridiagonal and coincide with the family $\Delta_{n}(A)$ given by \eqref{eq5}.

To prove the reverse implication, we need the following lemma.

\begin{lem}\label{lem8}
Let $R_{n}(A)$ be the operator defined by \eqref{eq12}. Then
$$
w(A) \leqslant \frac{n+1}{n} w\left(R_{n}(A)\right).
$$
\end{lem}

\begin{proof}
Let $x$ be any unit vector in $\mathcal{H}.$ Then
$$\underline{x}=\frac{1}{\sqrt{n+1}}\left[\begin{array}{c}x \\ x \\ \vdots \\ x\end{array}\right] \quad(n+1\textrm{ copies of }x)$$
is a unit vector and
$$
\left\langle\underline{x}, R_{n}(A) \underline{x}\right\rangle=\frac{n}{n+1}\langle x, A x\rangle.
$$
So,
$$
\begin{aligned}
w(A) &=\sup_{\|x\|=1}|\langle x, A x\rangle| \leqslant \frac{n+1}{n} \sup_{\|x\|=1}\left\langle\underline{x}, R_{n}(A) \underline{x}\right\rangle \\
& \leqslant \frac{n+1}{n}w\left(R_{n}(A)\right).
\end{aligned}
$$
\end{proof}
\vskip.1in
\noindent{\bf Proof of (ii) $\Rightarrow$ (i) in Theorem \ref{thm2}}:
For each $t \in \mathbb{R}$ let
$D(t)$ be the block-diagonal matrix
$$
D(t)=\operatorname{diag}\left(\e^{\imath  t} I, \e^{2 \imath t} I, \ldots, \e^{(n+1) \imath t} I\right) \text {. }
$$
Let $\Delta_{n}(A)$ be the matrix \eqref{eq5}. Then
$$
D(t) \Delta_{n}(A) D(t)^{*}=\Delta_{n}\left(\e^{\imath  t} A\right) .
$$
The matrix $D(t)$ is unitary. So, if $\Delta_{n}(A)$ is positive, then so is $\Delta_{n}\left(\e^{\imath  t} A\right)$ for every $t.$ In other words, we have the inequality
$$
\e^{-\imath  t} R_{n}(A)^{*}+2 I+\e^{\imath  t} R_{n}(A) \geqslant 0 \text {. }
$$
By Lemma \ref{lem7}, this is equivalent to the condition $w\left(R_{n}(A)\right) \leqslant 1.$ If this is true for all $n,$ then by Lemma \ref{lem8} we must have $w(A) \leqslant 1.$

This completes the proof of Theorem 2 . 

  \section{The power inequality and dilation theorems}

Unlike the norm $\|A\|,$ the
numerical radius is not submultiplicative.
The inequality
$$
w(A B) \leqslant w(A) w(B)
$$
can fail to hold even when $A$ and $B$ commute.
Worse still it can fail to hold even when $A$ and $B$ are powers of the same operator. This is shown by the following example.
Let $A$ be the $4 \times 4$ matrix
$$
A=\left[\begin{array}{llll}
0 & 0 & 0 & 0 \\
1 & 0 & 0 & 0 \\
0 & 1 & 0 & 0 \\
0 & 0 & 1 & 0
\end{array}\right] \text {. }
$$
A calculation shows that $w(A)=\frac{1+\sqrt{5}}{4}$ while $w\left(A^{2}\right)=w\left(A^{3}\right)=\frac{1}{2}.$ So $w(A) w\left(A^{2}\right)<\frac{1}{2}=w\left(A^{3}\right).$

However, we do have
\begin{equation}
w\left(A^{m}\right) \leqslant w(A)^{m}\label{eq19}
\end{equation}
for all operators $A$ and all positive integers $m.$
This is called the {\it power inequality} or {\it Berger's inequality}, (see \cite{bs}). Theorem \ref{thm2} leads to a quick transparent proof of this.
By homogenity, for any $A$ the operator $A/w(A)$ has numerical radius $1.$
So to prove \eqref{eq19} it suffices to show that $w\left(A^{m}\right) \leqslant 1$ whenever $w(A) \leqslant 1.$ This can be proved using the equivalence of conditions
(i) and (iii) in Theorem \ref{thm2}.
Just note that for every pair of integers $m$ and $n,$
the matrix $\Gamma_{n}\left(A^{m}\right)$ is a principal submatrix of the matrix $\Gamma_{r}(A)$ for some $r.$
Every principal submatrix of a positive matrix is positive.

The power inequality has an interesting history. In a famous paper on stability of numerical schemes for solving partial differential equations, Lax and Wendroff \cite{lw} showed that if $A$ is an operator on a $d$-dimensional space $\mathcal{H},$ then the condition $w(A) \leqslant 1$ implies that the operator $A$ is {\it power bounded}; i.e., there exists a constant $M$ such that $\left\|A^{m}\right\| \leq M$ for all $m.$
The constant $M$ found by Lax and Wendroff depends on $d$ and goes to infinity with $d.$
In an effort to extend this result to infinite-dimensional spaces, Halmos conjectured that the power inequality\eqref{eq19} holds for all operators $A$ on a Hilbert space of any dimension, and noted that a consequence of this would be that $w(A) \leq 1$ implies $\left\|A^{m}\right\| \leq 2$ for all $m.$ (See the relation \eqref{eq3}.)
The conjecture of Halmos was proved by Berger \cite{b} and \cite[p.52]{nfbk}.
Multiplicative properties of the numerical radius have been a major theme of research. See, e.g., \cite{hol, wg} and the many references given there.

Theorems \ref{thm1} and \ref{thm2} lead to two famous {\it dilation theorems}.
Let $A$ be an operator on $\mathcal{H}$ with $\|A\| \leq 1.$
The {\it Sz-Nagy dilation theorem} says that there is a Hilbert space $\mathcal{K}$ containing $\mathcal{H}$ and a unitary operator $U$ on $\mathcal{K}$ such that for every positive integer $n$ we have
\begin{equation}
A^{n}=\mathcal{P}_{\mathcal{H}} \left.U^{n}\right|_{\mathcal{H}}.\label{eq20}
\end{equation}
Here $\left.U^{n}\right|_{\mathcal{H}}$ denotes the restriction of the operators $U^{n}$ to $\mathcal{H}$ and $P_{\mathcal{H}}$ the orthogonal projection in $\mathcal{K}$ onto $\mathcal{H}.$
The operator $U$ is called a {\it unitary dilation} of $A.$

The relation \eqref{eq20} may be described in another way.
Let $\mathcal{H}$ and $\mathcal{K}$ be Hilbert spaces.
A linear map $V$ from $\mathcal{H}$ into $\mathcal{K}$
is called an {\it isometry} if $\langle Vx,Vy\rangle=\langle x,y\rangle$
for all $x,y\in\mathcal{H}.$
The adjoint $V^*$ then maps $\mathcal{K}$ onto $\mathcal{H}$ and $V^*V=I_\mathcal{H},$
the identity operator on $\mathcal{H}.$
If $B$ is any operator on $\mathcal{K},$
then the operator $A=V^*BV$ is called a {\it compression} of $B$ onto $\mathcal{H}.$
If we identify $\mathcal{H}$ as a subspace of $\mathcal{K},$
then in the decomposition $\mathcal{K}=\mathcal{H}\oplus\mathcal{H}^\perp,$
$B$ has a matrix representation
$B=\begin{bmatrix}A & *\\
* & *\end{bmatrix}$
in which $A$ is the top left block.
We call $B$ a {\it dilation} of $A.$
The Sz.-Nagy dilation theorem says that if $A$ is an operator on $\mathcal{H}$ such that $\|A\|\le 1,$
then there is a unitary operator $U$ on a Hilbert space $\mathcal{K}$ containing $\mathcal{H},$
such that for every $n,$
the operator $A^n$ is a compression of $U^n$ to $\mathcal{H}.$

The {\it Berger dilation theorem} says that if $A$ is an operator on $\mathcal{H}$ with $w(A) \leqslant 1,$ then there exists a Hilbert space $\mathcal{K}$ containing $\mathcal{H}$ and a unitary operator $U$ on $\mathcal{K}$ such that for every positive integer $n$ we have
\begin{equation}
A^{n}=\left.\mathcal{P}_{\mathcal{H}} 2 U^{n}\right|_{\mathcal{H}}.\label{eq21}
\end{equation}
Such a $U$ is called a unitary 2-dilation of $A.$

The positivity of the matrices \eqref{eq4} and \eqref{eq6} plays a crucial role in the construction of these dilations.
Let us briefly explain the connection.

Let $G$ be a group and let $\mathcal{H}$ be a complex Hilbert space.
Let $T(s)$ be a function on $G$ whose values are operators on $\mathcal{H}.$
Then $T$ is said to be a {\it positive definite function} if for every finitely supported function $u(s)$ from $G$ into $\mathcal{H}$
we have
$$
\sum_{s,t\in G}\langle u(s),T(s^{-1}t)u(t)\rangle \ge 0.$$
In block matrix terms, this condition means
that for any $N$ and for every choice of points $s_1,\ldots, s_N$ in $G$
the $N\times N$ block matrix with $i,j$ entry $T(s_i^{-1}s_j)$ is positive.

A {\it unitary representation} of $G$ is a function $U(s)$ on $G$ whose values are unitary operators on a Hilbert space $\mathcal{K},$
and they obey the rules $U(e)=I$ and $U(st)=U(s)U(t).$
(Here $e$ is the identity of the group $G$).

There is a connection between positive definite functions and unitary representations given by a theorem of Naimark.
If $U(s)$ is a unitary representation of $G$ in the space $\mathcal{K}$
and if $\mathcal{H}$ is a subspace of $\mathcal{K},$
then $T(s)=P_{\mathcal{H}}\left.U(s)\right|_{\mathcal{H}}$ is a positive definite function on $G,$
and $T(e)=I_{\mathcal{H}}$ (the identity operator on $\mathcal{H}$).
Conversely, if $T(s)$ is a positive definite function on $G$
whose values are operators on $\mathcal{H}$ and $T(e)=I_{\mathcal{H}},$
then their exists a unitary representation of $G$ on a Hilbert space $\mathcal{K}$
containing $\mathcal{H}$ as a subspace,
such that $T(s)=P_{\mathcal{H}}\left.U(s)\right|_{\mathcal{H}}.$

Now given an operator $A$ on $\mathcal{H}$
define a doubly-infinite sequence $\underline{A}(n)$ as follows:
\begin{equation}
\underline{A}(n)=\begin{cases}
A^n & \textrm{ if }n\ge 0\\
A^{*{-n}} & \textrm{ if }n<0.
\end{cases}\label{eq21a}
\end{equation}
Then the positivity of all matrices \eqref{eq4} can be translated to the statement that $\underline{A}(n)$ is a positive definite function
on the group $\mathbb{Z}.$
By Naimark's theorem there exists a unitary representation $U(n)$
of $\mathbb{Z}$ in a Hilbert space $\mathcal{K}$ containing $\mathcal{H}$
such that $\underline{A}(n)=P_{\mathcal{H}}\left.U(n)\right|_{\mathcal{H}}.$
But if $U(1)=U,$ then $U(n)=U^n$ (by the definition of a unitary representation).
This establishes the Sz.-Nagy dilation theorem \eqref{eq20}.
Similarly, the Berger dilation theorem \eqref{eq21}
can be derived from the positivity of the matrices \eqref{eq6}.
See Chapter 1 of the classic \cite{nfbk} for more details.

A typical application of the Sz.-Nagy dilation theorem is a simple and transparent proof of the {\it von Neumann inequality} -
a prominent result in operator theory.
This says that if $p$ is a polynomial and $A$ an operator with $\|A\|\le 1,$
then
\begin{equation}
\|p(A)\|\le\sup\limits_{|z|\le 1}|p(z)|.\label{eqb1}
\end{equation}
Let $\sigma(A)$ denote the spectrum of $A.$
If $U$ is a unitary operator,
then $p(U)$ is a normal operator and 
$$\sigma\left(p(U)\right)=p\left(\sigma(U)\right)\subseteq \{p(z):|z|=1\}.$$
Hence $\|p(U)\|\le \sup\limits_{|z|\le 1}|p(z)|.$
If $A$ is any operator with $\|A\|\le 1,$
then by the Sz.-Nagy dilation theorem $p(A)$ is a compression of $p(U)$ for some unitary operator $U.$
So the inequality \eqref{eqb1} holds.

If $A$ has a unitary dilation $U,$
then $\|A\|\le \|U\|=1.$
On the other hand, if $A$ has a unitary $2$-dilation as in \eqref{eq21},
then it is not obvious that $w(A)\le 1.$
This is, in fact, true.
We outline a proof using block matrices.

Suppose $A^n=2V^*U^nV,$ $n=1,2,\ldots,$
for some isometry $V$ and unitary $U.$
Let $E=\begin{bmatrix}E_{ij}\end{bmatrix}$ be the $(n+1)\times (n+1)$ block matrix with all its blocks $E_{ij}=I.$
Then $E=\frac{1}{n+1}E^2,$ and hence $E$ is positive.
Let $D$ be the block-diagonal matrix $D=\diag\left(I, U, U^2,\ldots,U^{n}\right).$
Then $2DED^*$ is also positive.
Note that
$$2DED^*=2\begin{bmatrix}I & U^* & \cdots &U^{*n}\\
U & I & \cdots & U^{*n}\\
\vdots & \vdots & \vdots\vdots\vdots & \vdots\\
U^{n}& U^{n-1} & \cdots & I\end{bmatrix}.$$
Now if $\underline{V}=\diag\left(V,V,\ldots,V\right),$
then $2\underline{V}^*\left(DED^*\right)\underline{V}$ is the matrix $\Gamma_n(A)$ in \eqref{eq6}.
Our argument shows this is positive for all $n,$ and hence $w(A)\le 1$ by Theorem \ref{thm2}.

\section{Proof of Theorem \ref{thm3}}

Lemma \ref{lem7} characterises operators with $w(A) \leq 1$ in terms of positivity of the expression \eqref{eq18}.
A part of our proof of Theorem \ref{thm2} depended on the Herglotz theorem that gives conditions for such positivity in terms of certain Toeplitz matrices.
For our proof of Theorem \ref{thm3} we need a different characterisation of positivity of expressions like \eqref{eq18}.

A {\it Laurent polynomial} is an expression of the form
$$
q(z)=\sum_{m=-n}^{n} a_{m} z^{m} .
$$
The classical Fejer-Riesz theorem from complex analysis says that if $q$ is a Laurent polynomial which takes positive values on the unit circle; i.e.,
$$
q\left(\e^{\imath  t}\right) \geqslant 0 \quad \text { for all } t \text {, }
$$
then there exists an ordinary polynomial
$$
p(z)=\sum_{m=0}^{n} b_{m} z^{m},
$$
such that
$$
q\left(\e^{\imath  t}\right)=\left|p\left(\e^{\imath  t}\right)\right|^{2} \quad \text { for all } t .
$$
See, e.g. \cite[p.20]{gs} and \cite[p.26]{simon}. There is an operator
version of this due to M. Rosenblum in which the complex coefficients $a_{m}$ and $b_{m}$ are replaced by operators.

\begin{thm}\label{thm9}
Let
\begin{equation}
Q(z)=\sum_{m=-n}^{n} A_{m} z^{m},\label{eq22}
\end{equation}
where $A_{m}$ are operators on $\mathcal{H}.$ Suppose
\begin{equation}
Q\left(\e^{\imath  t}\right) \geqslant 0 \quad \text { for all } t \in \mathbb{R}.\label{eq23}
\end{equation}
Then there exists a polynomial
\begin{equation}
P(z)=\sum_{m=0}^{n} B_{m} z^{m} ,\label{eq24}
\end{equation}
such that
$$
Q\left(\e^{\imath  t}\right)=P\left(\e^{\imath  t}\right)^{*} P\left(\e^{\imath  t}\right) \text { for all } t.
$$
\end{thm}
(See \cite{ros, rosrov,dr}).

Now let $w(A) \leqslant 1$ and consider the Laurent polynomial
$$
Q(z)=A^{*} z^{-1}+2 I+A z.
$$
Then by Lemma \ref{lem7}, $Q\left(\e^{\imath  t}\right) \geqslant 0$ for all $t.$
So, by Theorem \ref{thm9}, there exists a degree-one polynomial $P(z)=X+Y z$ such that
$$
Q\left(\e^{\imath  t}\right)=P\left(\e^{\imath  t}\right)^{*} P\left(\e^{\imath  t}\right).
$$
In other words
$$
\begin{aligned}
\e^{-\imath  t} A^{*}+2 I+\e^{\imath  t} A &=\left(X+\e^{\imath  t} Y\right)^{*}\left(X+\e^{\imath  t} Y\right) \\
&=X^{*} X+Y^{*} Y+\e^{-\imath  t} Y^{*} X+\e^{\imath  t} X^{*} Y.
\end{aligned}
$$
Since this is true for all $t,$ we must have
\begin{equation}
X^{*} X+Y^{*} Y=2 I \text { and } X^{*} Y=A .\label{eq25}
\end{equation}
Change $X$ and $Y$ to $\sqrt{2} X$ and $\sqrt{2} Y,$ respectively.
This shows the implication (i) $\Rightarrow$ (ii) of Theorem \ref{thm3}.

It is much easier to establish the reverse implication. Given any $X$ and $Y$ we have $(X-Y)^{*}(X-Y) \geqslant 0.$ It follows that
$$
2 \operatorname{Re} X^{*} Y \leqslant X^{*} X+Y^{*} Y.
$$
So, if $X, Y$ and $A$ obey the conditions
\eqref{eq7}, then $\operatorname{Re}A \leqslant I.$
Changing $Y$ to $\e^{\imath  t} Y$ we see that \eqref{eq7} implies
$\operatorname{Re}\, \e^{\imath  t} A \leqslant I$ for all $t.$
Hence, by Lemma \ref{lem7} , $w(A) \leqslant 1.$

We now prove the implication (ii) $\Rightarrow$ (iii) restricting ourselves to the case when $\mathcal{H}$ is finite-dimensional. (The theorem is true for infinite-dimensional spaces as well but the proof needs more elaborate arguments.) Let $X$ and $Y$ be as in \eqref{eq7}.
Then $X^{*} X=I-Y^{*} Y.$ So $\left(X^{*} X\right)^{1 / 2}=\left(I-Y^{*} Y\right)^{1/2},$ and $X$ has a polar decomposition $X=U\left(I-Y^{*} Y\right)^{1 / 2}$ where $U$ is a unitary operator.
We can then write
$$
A=2 X^{*} Y=2\left(I-Y^{*} Y\right)^{1 / 2} U^{*} Y.
$$
If we put $C=U^{*} Y,$ then $C^{*} C=Y^{*} Y,$ and $A=2\left(I-C^{*} C\right)^{1 / 2} C,$ which is the assertion \eqref{eq8}.

Conversely, assume \eqref{eq8}.
Put $X=\left(I-C^{*} C\right)^{1 / 2}$ and $Y=C.$ Then $X^{*} X+Y^{*} Y=I.$ So (iii) $\Rightarrow$ (ii).

(When $\mathcal{H}$ is infinite-dimensional, the polar decomposition has a partial isometry $U$ instead of a unitary. The argument we have given here then needs modifications.).

If $X, Y$ are any two operators, then\useshortskip
\begin{equation}
\left[\begin{array}{ll}
Y^{*} Y & Y^{*} X \\
X^{*} Y & X^{*} X
\end{array}\right]=\left[\begin{array}{ll}
Y^{*} & 0 \\
X^{*} & 0
\end{array}\right]\left[\begin{array}{ll}
Y & X \\
0 & 0
\end{array}\right] \geqslant 0 .\label{eq26}
\end{equation}
Assume $X, Y, A$ satisfy the relations \eqref{eq25}.
If we put $Y^{*} Y=I+H,$ then $X^{*} X=I-H$ and
the inequality \eqref{eq26} translates to
$$
\left[\begin{array}{cc}
I+H & A^{*} \\
A & I-H
\end{array}\right] \geqslant 0
$$
This proves the implication (ii) $\Rightarrow$ (iv). To complete
the proof we show that the condition (iv) of Theorem \ref{thm3} implies the condition (ii) of Theorem \ref{thm2}.
We choose a proof that shows the block matrix technique at its best.

Assume the condition (iv) of Theorem \ref{thm3}.
Note that if the matrix in \eqref{eq9} is positive,
then so are its diagonal entries $I \pm H.$
Let $\Delta_{n}(A)$ be the tridiagonal matrix in \eqref{eq5}.
We split it into a sum in which each summand has just one $2\times 2$ principal submatrix which is nonzero. That is:
\begin{align*}
\Delta_n(A) &=\begin{bmatrix}
2I & A^* & 0 & \cdots & 0\\
A & I-H & 0 & \cdots & 0\\
0 & 0 & 0 & \cdots 0\\
\vdots & \vdots & \vdots & \vdots\vdots\vdots & \vdots\\
0 & 0 & 0 & \cdots 0\end{bmatrix}\nonumber\\
& + \begin{bmatrix}
0 & 0 & 0 & 0 & \cdots & 0\\
0 & I+H & A^* & \cdots & 0\\
0 & A & I-H & \cdots & 0\\
\vdots & \vdots & \vdots & \vdots\vdots\vdots & \vdots\\
0 & 0 & 0 & \cdots & 0\end{bmatrix}
\end{align*}
\begin{align}
\ \ & +\begin{bmatrix}
0 & 0 & 0 & 0 & \cdots & 0\\
0 & 0 & 0 & 0 & \cdots & 0\\
0 & 0 & I+H & A^* & \cdots & 0\\
0 & 0 & A & I-H & \cdots & 0\\
\vdots & \vdots & \vdots & \vdots & \vdots\vdots\vdots & \vdots\\
0 & 0 & 0 & 0 & \cdots & 0\end{bmatrix}\nonumber\\
& +\cdots +\begin{bmatrix}
0 & 0 & \cdots 0 & 0\\
\vdots & \vdots & \vdots\vdots\vdots & \vdots & \vdots\\
0 & 0 & \cdots & 0 & 0\\
0 & 0 & \cdots & I+H & A^*\\
0 & 0 & \cdots & A & 2I\end{bmatrix}.\label{eq27}
\end{align}
Leave aside for a moment the first and the last summands in \eqref{eq27}.
Since $\begin{bmatrix} I+H & A^*\\
A & I-H\end{bmatrix}$ is positive,
all other summands are positive.
The top two $2\times 2$ block of the first summand in \eqref{eq27}
is 
$$
\begin{bmatrix}2I & A^*\\
A & I-H\end{bmatrix}=\begin{bmatrix}I+H & A^*\\
A & I-H\end{bmatrix}+\begin{bmatrix} I-H & 0\\
0 & 0\end{bmatrix},
$$
and being the sum of two positive matrices is positive.
The bottom $2\times 2$ block of the last summand in \eqref{eq27} is
$$
\begin{bmatrix}I+H & A^*\\
A & 2I\end{bmatrix}=\begin{bmatrix} I+H & A^*\\
A & I-H\end{bmatrix}+\begin{bmatrix}0 & 0\\
0 & I+H\end{bmatrix},
$$
and by the same argument this too is positive.
Thus $\Delta_n(A)$ being a sum of positive matrices is positive.
The proof of Theorem \ref{thm3} is complete.
\qed
\vskip.1in
The equivalence of (i), (iii) and (iv) of Theorem \ref{thm3}
was shown by Ando \cite{a}.
He first establishes (i)$\Rightarrow$ (iv) and then
(iv) $\Rightarrow$ (iii).
The arguments in this paper are based on dilation theory.
In a later report \cite{a1} Ando shows (i) $\Leftrightarrow$ (iv) using Arveson's extension theorem
from the theory of completely positive maps.
This proof is reproduced in \cite{rbh1} Theorem 3.5.1.
The proof we have given here is different from either of the two approaches,
and from some others such as \cite{bunce} and \cite{mathias}.

With our emphasis on block matrices we give a reformulation of one of the conditions in Theorem \ref{thm3} as follows.

\begin{prop}\label{propb1}
Let $A$ be an operator on $\mathcal{H}.$
Then the following two conditions are equivalent:
\begin{itemize}
\item[(i)] $w(A)\leqslant 1.$

\item[(ii)] There is an isometry $V$ from $\mathcal{H}$ into $\mathcal{H}\oplus\mathcal{H}$ such that
\begin{equation}
A=V^*\begin{bmatrix}0 & 2I\\
0 & 0\end{bmatrix}V.\label{eqb2}
\end{equation}
\end{itemize}
\end{prop}

\begin{proof}
Let $X$ and $Y$ be the operators in \eqref{eq7} and let $V=\begin{bmatrix}X\\Y\end{bmatrix}.$
Then $V^*V=X^*X+Y^*Y=I,$
and $V$ is an isometry.
Further
$$V^*\begin{bmatrix}0 & 2I\\
0 & 0\end{bmatrix}V=2X^*Y=A.$$
Thus (i)$\implies$(ii).

It can be easily seen that $w\left(\begin{bmatrix}0 & 2I\\
0 & 0\end{bmatrix}\right)=1$
and if $V$ is an isometry, then $w(V^*XV)\le w(X)$ for every $X.$
Thus (ii)$\implies$(i).
\end{proof}

\section{Asides and connections}

It is illuminating to consider the case of matrices \eqref{eq5} when $A$ is replaced by a complex number $a$:
$$\Delta_n(a)=\begin{bmatrix}2 & \overline{a} & 0 & \cdots & 0\\
a & 2 & \overline{a} & \cdots & 0\\
0 & a & 2 & \cdots & 0\\
\vdots & \vdots & \vdots & \vdots\vdots\vdots & \vdots\\
0 & 0 & 0 & \cdots & 2\end{bmatrix}.$$
The special case $a=-1$
$$\Delta_n(-1)=\begin{bmatrix}2 & -1 & 0 & \cdots & 0\\
-1 & 2 & -1 & \cdots & 0\\
0 & -1 & 2 & \cdots & 0\\
\vdots & \vdots & \vdots & \vdots\vdots\vdots & \vdots\\
0 & 0 & 0 & \cdots & 2\end{bmatrix}$$
represents the discretised version of the second order differential operator $-\frac{\operatorname{d}^2}{\operatorname{d}x^2}.$
This matrix is studied intensively in numerical analysis,
differential equations, Fourier analysis and several other contexts.
The eigenvalues of $\Delta_n(a)$ turn out to be
$$\lambda_j=1+|a|\cos\frac{j\pi}{n+2},\ \ 1\le j\le n.$$
The condition that all matrices $\Delta_n(a)$ be positive is equivalent to saying that all these numbers be positive for all $n$ and $j.$
This is equivalent to the condition $|a|\le 1.$

In tune with the themes of this paper we present another theorem
linking the condition $w(A)\le 1$ with positivity.
A trigonometricc polynomial is an expression of the form
\begin{equation}
g\left(\e^{\imath t}\right)=\sum_{j=-N}^{N} c_j\e^{\imath jt},\ \ -\pi\le t\le \pi.\label{a1}
\end{equation}
If $g\left(\e^{\imath t}\right)\ge 0$ for all $t,$
then $g$ is said to be {\it positive}.
Given an operator $A$ on $\mathcal{H},$ associate with a trigonometric polynomial $g$ the operator
\begin{equation}
\Phi_A(g)=\sum_{j=-N}^{N}c_j\underline{A}(j)+c_0I,\label{eqa2}
\end{equation}
where $\underline{A}(j)$ is the sequence defined by \eqref{eq21a}.

\begin{thm}\label{thm10}
The following two conditions are equivalent:
\begin{itemize}
\item[(i)] $w(A)\le 1.$
\item[(ii)] For every positive trigonometric polynomial $g,$
the operator $\Phi_A(g)$ defined by \eqref{eqa2}
is positive.
\end{itemize}
\end{thm}

\begin{proof}
Let $w(A)\le 1,$ and let $g$ be a positive trigonometric polynomial.
By the Fejer-Riesz Theorem cited at the beginning of Section 4,
$g$ must be of the form
$$g\left(\e^{\imath t}\right)=\sum_{j,k=0}^{N}\alpha_j\overline{\alpha_k}\, \e^{\imath (j-k)t}.$$
Then
\begin{equation}
\Phi_A(g)=\sum_{j,k=0}^{N}\alpha_j\overline{\alpha_k}\underline{A}(j-k)+|\alpha_0|^2I.\label{eqa3}
\end{equation}
We have to show this is a positive operator.
This means that for every vector $x$ in $\mathcal{H}$
we must have
$$\langle x,\Phi_A(g)x\rangle \ge 0.$$
Using the expression \eqref{eqa3}
this can be stated as
\begin{equation}
\sum_{j,k=0}^{N}\alpha_j\overline{\alpha_k}\langle x,\underline{A}(j-k)x\rangle +|\alpha_0|^2\|x\|^2\ge 0.\label{eqa4}
\end{equation}
If we put
$$y=\begin{bmatrix}
\alpha_0x\\
\alpha_1x\\
\vdots\\
\alpha_nx\end{bmatrix}$$
and let $\Gamma_n(A)$ be the matrix in \eqref{eq6},
then the condition \eqref{eqa4} can be expressed as
$$\langle y,\Gamma_n(A)y\rangle\ge 0.$$
This is true because $\Gamma_n(A)$ is a positive operator by Theorem \ref{thm2}.

To prove the converse, suppose $\Phi_A$ satisfies condition (ii).
Fix a $\theta$ in $[-\pi,\pi]$
and let
$$p\left(\e^{\imath t}\right)=\frac{1}{\sqrt{n+1}}\sum_{j=0}^{N}\e^{\imath j\theta}\e^{\imath jt}.$$
Then
$$
g\left(\e^{\imath t}\right)=\left|p\left(\e^{\imath t}\right)\right|^2=\frac{1}{n+1}\sum_{j=-N}^{N}\left(N+1-|j|\right)\e^{\imath (\theta+t)}
$$
is a positive trigonometric polynomial.
So, the operator
$$
\Phi_A(g)=\sum_{j=-N}^{N}\frac{N+1-|j|}{N+1}\e^{\imath j\theta}\underline{A}(n)+I$$
is positive.
This is true for all $\theta.$
Hence, by Lemma \ref{lem6},
the matrices
\begin{equation}
\frac{1}{N+1}\begin{bmatrix}2I & NA^* & (N-1)A^{*2} & \cdots & A^{*N}\\
NA & 2I & NA^* & \cdots & 2A^{*(N-1)}\\
\vdots & \vdots & \vdots & \vdots\vdots\vdots & \vdots\\
A^N & 2A^{N-1} & 3A^{N-2} & \cdots & 2I\end{bmatrix},\label{eqa5}
\end{equation}
are all positive.
$($The matrix above is the Hermitian Toeplitz matrix with first row
$\left.\left(2I ,NA^*,(N-1)A^{*2},\ldots,2A^{*(N-1)}, A^{*N}\right).\right).$
For each $n<N,$ consider the top $n\times n$ principal submatrix of \eqref{eqa5}.
Keeping $n$ fixed, let $N\to\infty.$
This shows that the matrices $\Gamma_n(A)$ in \eqref{eq6} are positive.
Hence $w(A)\le 1.$
\end{proof}

The usual proof of Theorem \ref{thm10} (see, e.g. \cite[Theorem 3.15]{paulsen})
invokes "complete positivity" of the map $\Phi_A.$
We have bypassed that argument by using block matrix techniques only.

\vskip.1in
\section*{Conflict of interest}

 The authors declare that they have no conflict of interest.

\end{document}